\documentclass{amsart}
\usepackage{amsmath,amsthm,amssymb,amsfonts}
\usepackage[colorlinks=true]{hyperref}
\usepackage{color}

\hypersetup{urlcolor=blue, citecolor=blue}

\def\doi#1{   {\href{http://dx.doi.org/#1}
   {{\mdseries\ttfamily DOI}}}}

\newcommand{\de}{\delta}    
  \newcommand{\ep}{\varepsilon}

\newcommand{\R}{\mathbb{R}}\newcommand{\Z}{\mathbb{Z}}
\newcommand{\N}{\mathbb{N}}

\newcommand{\pt}{\partial_t}\newcommand{\pa}{\partial}
\newcommand{\les}{{\lesssim}}
\newcommand{\beeq}{\begin{equation}}\newcommand{\eneq}{\end{equation}}

\newcommand{\Sp}{{\mathbb S}}

\newcommand{\supp}{\text{supp}}
\newcommand{\cd}{{\,\cdot\,}}

\newenvironment{prf}{\noindent {\bf Proof.} }{\endprf\par}
\def \endprf{\hfill  {\vrule height6pt width6pt depth0pt}\medskip}

\numberwithin{equation}{section}

\def\O{{\mathcal{O}}}

\def\<{\langle}             \def\>{\rangle}
\def\({\left(}                 \def\){\right)}
\newcommand{\la}{\langle}
\newcommand{\ra}{\rangle}

\newtheorem{thm}{Theorem}[section]

\newtheorem{coro}[thm]{Corollary}
\newtheorem{lem}[thm]{Lemma}

\theoremstyle{remark}

\theoremstyle{definition}

\title
{The Strauss conjecture on asymptotically flat space-times}

\author{Jason Metcalfe}
\address{Department of Mathematics, University of North Carolina,
  Chapel Hill, NC  27599-3250, USA}
\email{metcalfe@email.unc.edu}
\urladdr{http://metcalfe.web.unc.edu}

\author{Chengbo Wang}
\address{School of Mathematical Sciences\\                Zhejiang University\\                Hangzhou 310027, China}
\email{wangcbo@gmail.com}
\urladdr{http://www.math.zju.edu.cn/wang}
\date{}

\thanks{
The first author was supported in part by NSF grant DMS-1054289.
The second author was supported in part by NSFC 11301478, National Support Program for Young Top-Notch Talents.
}
\dedicatory{} \commby{}

\begin{document}

\begin{abstract}
By assuming a certain 
localized energy estimate, we prove the existence portion of the
Strauss conjecture on asymptotically flat manifolds, possibly exterior
to a compact domain, when the spatial dimension is
$3$ or $4$. In particular, this result applies to the $3$ and $4$-dimensional 
Schwarzschild and Kerr (with small angular momentum) black hole backgrounds, 
long range asymptotically Euclidean spaces, and small time-dependent
asymptotically flat perturbations of Minkowski space-time.  We also
permit lower order perturbations of the wave operator.
The key estimates are a class of weighted
Strichartz estimates, which are used near infinity where the metrics
can be viewed as small perturbations of the Minkowski metric, and the assumed
localized energy estimate, which is used in the remaining compact set.
\end{abstract}



\maketitle 

\section{Introduction}
\label{sec-Intro}

The purpose of this article is to establish global existence for
semilinear wave equations with small initial data on a wide class of
space-times.  Given a space-time $(M,g)$, we shall examine
\begin{equation}\label{main.eq.0}
\begin{cases}
\Box_g u := \nabla^\alpha \partial_\alpha u = F_p(u),\quad (t,x)\in
M,\\
u(0,x)=u_0(x),\quad \partial_t u(0,x)=u_1(x),
\end{cases}
\end{equation}
for sufficiently nice initial data $u_0, u_1$.  Here
\begin{equation}
  \label{Fp}
  \sum_{0\le j\le 2} |u|^j |\partial_u^j F_p(u)|\lesssim |u|^p\quad
  \text{for } |u|\ll 1,
\end{equation}
and typical examples include $F_p(u)=\pm |u|^p$ and $F_p(u)=\pm |u|^{p-1}u$.
In the process, we
shall also weaken some hypotheses that were made on the data in \cite{LMSTW}.
We shall consider
asymptotically flat space-times that permit a localized energy
estimate.  This suffices in a large compact set where the influence of
the geometry is most significant.  The flatness then guarantees that we
are sufficiently close to the Minkowski space-time outside this
compact set to derive analogs of the estimates that have been
previously used, namely weighted Strichartz estimates.  This strategy has become quite common and is
found in, e.g., a large number of the references.  See, amongst
others, \cite{Burq}, \cite{DMSZ}, \cite{HMSSZ}, \cite{LMSTW},
\cite{MMTT}, \cite{MetTa07, MetTa09}, \cite{MTT}, \cite{Ta08, Ta13}.

Rather than considering \eqref{main.eq.0}, we shall instead examine
\begin{equation}\label{main.eq}
\begin{cases}
P u = F_p(u),\quad (t,x)\in
M\\
u(0,x)=u_0(x),\quad \partial_t u(0,x)=u_1(x),
\end{cases}
\end{equation}
where
\[Pu = \partial_\alpha g^{\alpha\beta}\partial_\beta u +
b^\alpha \partial_\alpha u + c u.\]
Here we shall be permitted to work with the volume form $dV = dx\, dt$.
A conjugation by $g^{-1/4}$ reduces \eqref{main.eq.0} to this case.
In the above, the common convention $t=x^0$ is employed.  The Einstein summation
convention is used, as well as the convention that Greek
indices $\alpha, \beta, \gamma$ range from $0$ to $n$ while Latin
indices $i, j, k$ will run from $1$ to $n$.  
We will use $\mu,\nu$ to
denote multi-indices.

Let us begin with the ambient space-time manifolds.  We shall set
$M=\R_+\times \R^n$ or $M=\R_+\times (\R^n\backslash \mathcal{K})$
where $\mathcal{K}\subset \{|x|<R_0\}$ for some $R_0>0$ and has a smooth boundary.  In the
case of a $(1+3)$-dimensional black hole space-time, $\mathcal{K}$
will be a ball of radius $R_0$.  
The manifold $M$ is equipped with a Lorentzian metric $g$
\[g = g_{\alpha\beta}(t,x) dx^\alpha dx^\beta\]
where $g_{\alpha\beta}\in C^3(M)$ has signature $(1,n)$ and inverse
$g^{\alpha\beta}$.  

Many of our estimates will rely on a spatial dyadic summation.   
So before proceeding, we introduce the notation
\[\|u\|_{\ell^s_q A} = \|\phi_j(x)u(t,x)\|_{\ell^s_q A } =
  \Bigl\|\Bigl(2^{js}\|\phi_j(x)u(t,x)\|_A\Bigr)\Bigr\|_{l^q_{j\ge
      0}},\]
for a norm $A$ and a partition of unity subordinate to the inhomogeneous dyadic
(spatial) annuli.  Thus,
\[\sum_{j\ge 0} \phi_j^2(x)=1,\quad \text{supp } \phi_j\subset \{\la
x\ra\approx 2^j\}.\]
Similarly, we use $\dot{\ell}^s_q$ to denote the
  homogeneous version.

{\bf Hypothesis 1: Asymptotic Flatness.}  We shall assume that $(g,
b ,c)$ is asymptotically flat in the following
sense.  We first assume that $g$ can be decomposed as
\begin{equation}\tag{H1}
  \label{H1}
  g = m + g_0(t,r)+g_1(t,x)
\end{equation}
where $m=\text{diag}(-1,1,\dots,1)$ denotes the Minkowski metric, 
$g_0$ is a {\em radial} long range perturbation, and $g_1$
is a short range perturbation.  More specifically, we assume
\begin{equation}\tag{H1.1}
  \label{H1.1}
  \|\partial^\mu_{t,x} g_{i,{\alpha\beta}}\|_{\ell^{i+|\mu|}_1
    L^\infty_{t,x}} = \O(1),\quad i=0,1,\quad |\mu|\le 3.
\end{equation}
  The long range perturbation is
radial in the sense that when writing out the metric $g$ with
$g_1=0$, in polar coordinates $(t,x)=(t,r\omega)$ with $\omega\in
\Sp^{n-1}$, we have
$$m+g_0(t,r)=\tilde g_{00}(t,r)dt^2+2\tilde g_{01}(t,r)dtdr+\tilde g_{11}(t,r)dr^2+\tilde g_{22}(t,r) r^2 d\omega^2.$$ 
In this form, the assumption \eqref{H1.1} is equivalent to the
following requirement
\begin{equation*}
  \|\partial_{t,x}^\mu (\tilde{g}_{00}+1, \tilde{g}_{11}-1,
  \tilde{g}_{22}-1,\tilde{g}_{01})\|_{\ell^{|\mu|}_1L^\infty_{t,x}}=\O(1),\quad
    |\mu|\le 3.
\end{equation*}
Moreover, for the lower order perturbations, we shall assume that
\begin{equation}
  \label{H1.2}\tag{H1.2}
   \|\partial^\mu_{t,x} b \|_{\ell^{1+|\mu|}_1 L^\infty_{t,x}} +
   \|\partial_{t,x}^\mu c\|_{\ell^{2+|\mu|}_1 L^\infty_{t,x}} =
   \O(1),\quad |\mu|\le 2.
\end{equation}

We note that these hypotheses are reminiscent of those that have
appeared previously in \cite{Ta08, Ta13}, \cite{MetTa07},
\cite{MTT}, and \cite{W}.  The radial symmetry on the long range perturbation is
primarily used to assist with commuting rotational vector fields
$\Omega$ with $P$
in the sequel.

{\bf Hypothesis 2: Localized energy.}  We shall assume that the metric
$g$ permits a uniform energy estimate and (weak) localized energy
estimate.  Specifically, we assume that there exists $R_1>R_0$ so that if $u$ is a solution to the
linear wave equation $P u = F$, then
\begin{multline}
  \label{H2}\tag{H2}
 \|\partial \partial^\mu u\|_{L^\infty_t L^2_x} +
  \|(1-\chi)\partial \partial^\mu u\|_{\ell^{-1/2}_\infty L^2_{t,x}} +
  \|\partial^\mu u\|_{\ell^{-3/2}_\infty L^2_{t,x}}
\\\lesssim \|u(0,\cd)\|_{H^{|\mu|+1}} + \|\partial_t
u(0,\cd)\|_{H^{|\mu|}} + \sum_{|\nu|\le |\mu|} \|\partial^\nu
F\|_{L^1_tL^2_x} 
\end{multline}
for all $|\mu|\le 2$.  Here $\chi$ is a smooth function that is
identically $1$ on $B_{R_1/2} :=\{|x|\le R_1/2\}$  and is supported
in $B_{R_1}$.
The next section will discuss this
hypothesis in further detail and will provide some examples where it
is known to hold.

The main result of this paper then states that solutions to \eqref{main.eq}
exist globally for sufficiently small Cauchy data.  See
Theorem~\ref{main_thm} for a more precise statement once further
notations are established.
\begin{thm}\label{metaTheorem0}
Let $n=3, 4$,
and assume \eqref{H1}, \eqref{H1.1}, \eqref{H1.2}, and
\eqref{H2}. 
Consider the problem \eqref{main.eq}
with $p>p_c$ where $p_c>1$ solves
\[(n-1)p_c^2 - (n+1)p_c - 2 = 0.\]
Then there exists a global solution $u$ for any initial data which are
sufficiently small, decaying, and regular.
\end{thm}

The first result of this type regarding such nonlinear wave equations (on Min\-kow\-ski
space) with small powers $p$ was \cite{John79} where global existence
above $1+\sqrt{2}$ and blow-up for powers below the same was
established in $n=3$.  In generic dimensions, blow-up for powers below
$p_c$ was established in \cite{Sideris}, while global existence for small
initial data when $p>p_c$ followed in \cite{GLS97}, \cite{Ta01-2}.

Several recent works have generalized these results to exterior
domains and asymptotically flat backgrounds in dimensions $n\le 4$.
This paper represents a unification and generalization of these
results, most of which will be outlined in the next section.  We refer
the interested reader to, e.g., \cite{WaYu11p} for a complete history.

A key estimate that has permitted such progress is a class of weighted
Strichartz estimates, which was developed independently in
\cite{HMSSZ} and \cite{FaWa}.  This class of estimate has subsequently
been shown to be robust under small, asymptotically flat
perturbations.  See, e.g., \cite{SoWa10}, \cite{WaYu11}, and
\cite{LMSTW}.  The localized energy estimates \eqref{H2} exactly permit
us to pass to the case of such a small perturbation by handling the
regions where the geometry has the most significant role.  This method
stems from \cite{LMSTW}.  

The key difference between the current result and \cite{LMSTW} is the
approach to proving the weighted Strichartz estimates.  In
\cite{LMSTW} a duality argument of \cite{WaYu11} is used to ``divide
through'' by a derivative in the localized energy estimate.  That
process required that the data be compactly supported.  Here we use
the more robust ideas of \cite{MetTa07}.

\subsection{Notations}  Before proceeding, let us set more
notations.  The vector fields to be used will be labeled as
\[Y=(Y_1,\dots, Y_{n(n+1)/2}) = \{\nabla_x,\Omega\},\quad
Z=\{\partial, \Omega\},\]
where $\partial = \partial_{t,x}$.  Here $\Omega$ denotes the
generators of spatial rotations:
\[\Omega_{ij} = x_i\partial_j - x_j\partial_i,\quad 1\le i<j\le n.\]
For a norm $X$ and a nonnegative integer $m$, we shall use the shorthand
\[|Z^{\le m} u| = \sum_{|\mu|\le m} |Z^\mu u|,\quad \|Z^{\le m} u\|_X
= \sum_{|\mu|\le m} \|Z^\mu u\|_X,\]
with the obvious modification for $\|Y^{\le m} u\|_X$, e.g.
Letting $L^q_\omega$ be the standard Lebesgue space on the sphere
$\Sp^{n-1}$, we will use the following convention for mixed norms
$L^{q_1}_tL^{q_2}_rL^{q_3}_\omega$:
\[\|f\|_{L^{q_1}_tL^{q_2}_rL^{q_3}_\omega(M)} = \Bigl\|\Bigl(\int
\|f(t,r\omega)\|^{q_2}_{L^{q_3}_\omega}
r^{n-1}\,dr\Bigr)^{1/q_2}\Bigr\|_{L^{q_1}(\{t\ge 0\})},\]
with trivial modification for the case $q_2=\infty$.  Occasionally,
when the meaning is clear, we shall omit the subscripts.

\section{Discussion of Hypothesis 2}

In this section, we discuss the localized energy estimates, which are
assumed to hold by hypothesis \eqref{H2}.  On $\R_+\times\R^n$, $n\ge 3$ with the
Minkowski metric, such estimates look like
\begin{equation}\label{kss}
\|u'\|_{L^\infty_tL^2_x\cap \ell_\infty^{-1/2} L^2_{t,x}} +
\|u\|_{\ell_\infty^{-3/2} L^2_{t,x}} \lesssim \|u'(0,\cd)\|_{L^2} +
\|\Box u\|_{L^1_tL^2_x + \ell_1^{1/2} L^2_{t,x}}
\end{equation}
and the $\ell_\infty^{-3/2}$ on the second term may be replaced by
$\ell_2^{-3/2}$ summability in dimensions $n\ge 4$.
Estimates of this form date back to \cite{Mo2} and can be proved by
integrating $\Box u$ against a multiplier of the form $f(r)\partial_r
u + \frac{n-1}{2} \frac{f(r)}{r}u$.  The estimate follows after
integration by parts and choosing $f(r) = r/(r+R)$, where $R$ is a
dyadic number.  See, e.g., \cite{MetSo06} for this proof and
\cite{MTT} for a more exhaustive history of such estimates.

When the background geometry permits trapped rays, which roughly
speaking are light rays that do not escape a compact set in finite
time, such estimates are known to be false.  See \cite{Ral} and \cite{Sbierski}.  In
certain situations, however, a related estimate with a loss can be
recovered.  One such possibility is to ask for a localized energy
estimate to hold upon localizing away from the trapping, and this is
what Hypothesis 2 represents.  See, e.g., \cite{MTT} and \cite{Ta13}
for situations where similar hypotheses were employed.

It is worth noting that \eqref{H2} is assuming three sorts of bounds.
The first is a uniform bound on the energy and corresponds to the
first term in the estimate.  The second term corresponds to the
localized energy when cut off away from any possible trapping.  The last term is a
lossless lower order term from the localized energy.  This term
corresponds to the second term in the left side of \eqref{kss}, and we
are assuming that the same bound can be obtained in the presence of
the geometry.  It is also assumed that higher order versions
(corresponding to $|\mu|\neq 0$) of each estimate hold.  It would
suffice to assume that the first and third terms are controlled in a
sufficiently large ball and that the second term (without the cutoff)
can be bounded in the dyadic region with radii comparable to the same
ball.  Using standard cutoff techniques, it is relatively easy to show
that these two assumptions are equivalent, and as such, we
choose to employ the more common form.

The hypothesis \eqref{H2} is known to hold in several cases, and we
provide a sampling of
these here.  In each case, one immediately obtains small data global
existence for \eqref{main.eq} as a corollary of Theorem
\ref{metaTheorem0}.

\subsection{Nontrapping asymptotically Euclidean manifolds}
In the case that $R_0=0$, $b=0$, $c=0$, and the manifold is a product manifold, in the
sense that
\[g = -dt^2 + g_{ij}(x)dx^idx^j,\]
and the metric is nontrapping, then
the localized energy estimates, at least when the summability in
\eqref{H1.1} is replaced by any additional power decay, were shown in
\cite{BoHa}.  
The resulting nonlinear result contained herein,
then, recovers and generalizes the $p>p_c$ portions of \cite{SoWa10}
and \cite{WaYu11}.  The localized energy estimate in this setting is generalized further in the forthcoming
result \cite{MST}, which permits manifolds of general structure, only
requires \eqref{H1.1} and \eqref{H1.2}, allows for the lower order terms (even complex
valued) provided that eigenvalues and resonances are not introduced,
and allows for some time dependence in the coefficients.

Similarly, even with the possibility of a boundary, for compactly
supported, time-independent, nontrapping perturbations of the
Euclidean Laplacian, the localized energy estimates were proved in
\cite[Theorem 3]{Burq}.  The result given there can easily be
supplemented with \eqref{kss} to drop the assumption of compact
support on the initial data.  Moreover, the higher order estimates
follow from standard elliptic estimates upon noticing that
$\partial_t$ commutes with the equation.  Our theorem, then, recovers
the corresponding nonlinear results of \cite{HMSSZ}.  This also
recovers a result of \cite{DMSZ} in the $n=4$ flat case exterior to a
star-shaped obstacle.

\subsection{Time-dependent, small, asymptotically flat manifolds}
In the case where the constants in \eqref{H1.1} and \eqref{H1.2} are sufficiently small,
the localized energy hypothesis follows from the same argument
outlined for proving \eqref{kss}.  See \cite{MetSo06, MetSo07} and
\cite{MetTa07}.  The same can also be proved exterior to a star-shaped
obstacle with Dirichlet boundary conditions \cite{MetSo06, MetSo07},
\cite{MetTa09}.  And as such, the existence portion of the Strauss
conjecture immediately follows from Theorem \ref{metaTheorem0}.

\subsection{Kerr space-times}
The hypothesis \eqref{H2} has been verified (in the $|\mu|=0$ case) on Schwarzschild
backgrounds in \cite{BS, BSerr}, \cite{DaRo, DaRo09}, \cite{MMTT}.
And in the setting of Kerr space-times with $a\ll M$, similar results
have been shown in \cite{TT}, \cite{AB09}, \cite{DaRo08, DaRoNotes,
  DaRoNew}.  See \cite{TT} for the higher order ($|\mu|\neq 0$) case.
In these settings, Theorem \ref{metaTheorem0} recovers the results of
\cite{LMSTW}, and in the case of Kerr, it relaxes some assumptions on
the data.  Here, we also note that \eqref{H2} was proved for a class
of small perturbations of such Kerr metrics in \cite{MTT}, and thus,
the nonlinear existence result extends to such cases as well.  We note
also that \cite{DaRoSR} essentially verifies \eqref{H2} for the full
subextremal case $|a|<M$.

\subsection{$(1+4)$-dimensional Myers-Perry space-times}
Recent studies have proved similar localized energy results on higher
dimensional black hole backgrounds.  These include the results of
\cite{LM}, \cite{Schlue} on (hyperspherical) Schwarzschild backgrounds and
\cite{LMTT} on Myers-Perry backgrounds with small angular momenta.
While these results are given for $|\mu|=0$, the techniques of \cite{TT}
can be mimicked to obtain the higher order estimate.  And as a result,
solutions to \eqref{main.eq} will exist globally for sufficiently
small data provided that $p>2$.  Here $p_c=2$ when the spatial
dimension is $4$.

\subsection{Surfaces of revolution with degenerate trapping}
An interesting class of surfaces of revolution where the generating
function has a unique degenerate minimum were introduced in
\cite{ChrWu}.  On such manifolds, it can be shown that the local
smoothing estimate for the Schr\"odinger equation necessitates an
algebraic loss of smoothness.  The $|\mu|=0$ analog of \eqref{H2} was
explored in \cite{Booth} (non-degenerate case) and \cite{Roddy}.  See
also the forthcoming work \cite{BCMP}.  Using elliptic estimates and the fact that the metrics are static, one
could easily obtain the full of \eqref{H2}.  So, again, a small data global
existence result for \eqref{main.eq} with $p>p_c$ follows
immediately.  The setup here is slightly different in that the
manifolds under consideration have two ends (both asymptotically
Euclidean), but the modifications that would be required are straightforward.

\subsection{Equations with lower order perturbations}
In \eqref{main.eq}, we also permit lower order perturbations of the
d'Alembertian.  A number of preceding results studied the case of the
flat Laplacian with such lower order perturbations, but here we allow
for full generality.  Moreover, the hypothesis \eqref{H2} generalizes
many of the conditions, such as a sign condition or a smallness
condition, on the perturbations.  See, e.g., \cite{ST} where small
potentials were examined; \cite{GHK} which allowed non-negative
compactly supported potentials; and \cite{V} where potentials with
some time-decay are permitted.  The case of damping represents a
fundamentally different problem for which there have been a number of
studies following the seminal work \cite{TY}, but we do not explore
these here.

\section{Sobolev-type estimates}
\label{sec-Sobolev}

In this section, we gather several Sobolev-type and trace estimates.  

To begin, we recall a standard trace estimate on the sphere.  See,
e.g., \cite[(1.3)]{FaWa}
\begin{lem}[Trace estimates]\label{thm-trace}
Let $n\ge 2$ and $1/2< s<n/2$, then
\beeq\label{eq-trace}
\|r^{n/2-s} u\|_{L^\infty_r H^{s-1/2}_\omega}\lesssim \|u\|_{\dot H^s}.
\eneq
\end{lem}

We shall also need the following variant of the Sobolev embeddings.
This has essentially been proved in \cite[Lemma 3.1]{LMSTW} in the
case $n=3$ and $\varepsilon=0$ and is akin to the original estimates
of \cite{Klainerman}.  In the case of $n=4$, extra care is
required so that the number of derivatives does not exceed $2$.

\begin{lem}[Weighted Sobolev estimates]\label{thm-Sobo}
Let $n\ge 2$ and $R\ge 3$. We have
\beeq\label{eq-Sobo}
\|r^{\beta}u\|_{L^q_{r} L^\infty_\omega(r\ge R+1)}\les \sum_{|\mu|\le [(n+1)/2]} \|r^{\beta-(n-1)/p+(n-1)/q} Y^\mu u\|_{L^p_{r}L^{2+\ep}_\omega(r\ge R)}
\eneq
for any $\beta\in\R$, $ 2+\ep\le p\le q\le \infty$ and $\ep>0$. When $\ep=0$, we need to require $|\mu|\le [(n+2)/2]$ instead.
Moreover, we have 
\beeq\label{eq-Sobo2}
\|r^{\beta}u\|_{L^q_{r} L^4_\omega(r\ge R+1)}\les \sum_{|\mu|\le [(n+3)/4]} \|r^{\beta-(n-1)/p+(n-1)/q} Y^\mu u\|_{L^p_{r}L^{2}_\omega(r\ge R)}
\eneq
for any $2\le p\le q\le 4$ and $\beta\in\R$.
\end{lem}

\begin{prf}
By Sobolev estimates on ${\mathbb R}\times \Sp^{n-1}$, we have 
for each $j\in {\mathbb N}$ the uniform bounds
$$\|v\|_{L^\infty_rL^\infty_\omega([j,j+1]\times \Sp^{n-1})}
\lesssim \sum_{|\mu|\le [(n+1)/2]}\Bigl(\int_{j-1}^{j+2}\int_{\Sp^{n-1}}|Y^\mu v|^{2+\ep} \, d\omega dr\Bigr)^{\frac1{2+\ep}}.
$$
Hence, for any $\beta\in\R$ and $j\ge 3$,
\beeq\label{ap-trace-local} 
\|r^\beta v\|_{L^\infty_r([j,j+1])L^\infty_\omega} \lesssim
  \sum_{|\mu|\le  [(n+1)/2]} \|r^{\beta-(n-1)/(2+\ep)} Y^\mu v\|_{L^{2+\ep}_r([j-1,j+2])L^{2+\ep}_\omega}.\eneq
The factor $r^{-(n-1)/(2+\ep)}$ on the right comes from the fact that the volume
element for $\R^n$ is $r^{n-1} dr d\omega$.  By H\"older's
inequality, it follows that for every $1\le q\le\infty$ and $p\ge 2+\ep$,
\beeq\label{ap-S3} \|r^\beta v\|_{L_r^{q}([j,j+1])L_{\omega}^{\infty}} \lesssim \sum_{|\mu|\leq  [(n+1)/2]} \|r^{\beta+\frac{n-1}{q}-\frac{n-1}{p}} Y^\mu v\|_{L_r^{p}([j-1,j+2])L_{\omega}^{2+\ep}} .\eneq
This yields the inequality \eqref{eq-Sobo} if we $l^p$-sum over $j\ge R+1$ using the
Minkowski integral inequality and the fact $p\le q$.  

Inequality \eqref{eq-Sobo2} follows from a similar argument.  The proof of \eqref{ap-trace-local} also yields
$$\|v\|_{L^4_r([j,j+1])L^4_\omega}\lesssim j^{-\frac{n-1}{4}}
\sum_{|\mu|\le [(n+3)/4]} 
\|Y^\mu v\|_{L^2_r([j-1,j+2])L^2_\omega},
$$
which implies \eqref{eq-Sobo2} after an application of H\"older's
inequality, weighting appropriately, and $l^p$-summing over $j$.
\end{prf}

\section{Weighted Strichartz estimates}
The key linear estimate near infinity is a weighted Strichartz
estimate akin to those introduced in \cite{HMSSZ} and \cite{FaWa}.  As
a corollary to our techniques, 
the same holds for small, asymptotically decaying
perturbations of the Minkowski metric.  In essence, this is what the
main theorem below states as 
the estimate is only applied outside of a large compact set
where the metric perturbation may be taken to be small due to the
asymptotic flatness.

\begin{thm}\label{thm-wStri} Let $n\ge 3$, and suppose that
  \eqref{H1}, \eqref{H1.1}, \eqref{H1.2}, and \eqref{H2} hold.  Then
  there exists $R>R_1$ so that if
  $\psi_R$ is identically $1$ on $B_{2R}^c$ and vanishes on $B_R$,
we have
\begin{multline}
  \label{weightedStrichartz}
  \|\psi_R Z^{\le 2} w\|_{{\ell}_p^{\frac{n}{2}-\frac{n+1}{p}-s}L^p_{t,r}
    H_\omega^{\sigma}}\lesssim \|w(0,\cd)\|_{H^3}
  + \|\partial_t w(0,\cd)\|_{H^2} \\+
  \|\psi_R  Y^{\le 2} w(0,\cd)\|_{\dot{H}^s} + \| \psi_R Y^{\le 2} \partial_t
  w(0,\cd)\|_{\dot{H}^{s-1}}  + \| \psi_R Z^{\le 1} Pw(0,\cd)\|_{\dot{H}^{s-1}}
\\+
  \|\psi_R^p Z^{\le 2} P
  w\|_{\ell_1^{-\frac{n-2}{2}-s}L^1_{t,r}H^{s-1/2}_\omega} +
  \|\partial^{\le 2} Pw\|_{L^1_t L^2_x}
\end{multline}
  for any $p\in (2,\infty)$, $s\in (1/2-1/p,1/2)$, and $0\le \sigma< \min(s-1/2+1/p,1/2-1/p)$.
\end{thm}

For clarity, as a corollary of the techniques of proof of Theorem
\ref{thm-wStri}, we state separately what would result for
sufficiently small
perturbations, which removes the need to cutoff to the exterior of
$B_R$, when there are no vector fields with which to contend.
Moreover, in the case of sufficiently small perturbations, the
hypothesis \eqref{H2} is known due to \cite{MetSo06}, \cite{MetTa07}.

\begin{coro}\label{cor-wStri} Let $n\ge 3$, and suppose that
$g^{\alpha\beta} = g^{\beta\alpha}\in C^2$, $b^\alpha\in C^1$, and
$c\in C$.
Then there exists $\delta_1>0$ so that if $(g,b,c)$ satisfies
\[\|g-m\|_{\dot{\ell}^0_1 L^\infty_{t,x}} + \|(\partial g,
b)\|_{\dot{\ell}^1_1 L^\infty_{t,x}} + \|(\partial^2 g,\partial
b, c)\|_{\dot{\ell}^2_1 L^\infty_{t,x}}\le \delta_1,\]
we have
\begin{multline}
  \label{weightedStrichartz1}
  \|w\|_{\dot{\ell}_p^{\frac{n}{2}-\frac{n+1}{p}-s}L^p_{t,r}
    H_\omega^{\sigma}}\lesssim
  \|w(0,\cd)\|_{\dot{H}^s} + \|\partial_t w(0,\cd)\|_{\dot{H}^{s-1}} \\+
  \|P w\|_{\dot{\ell}_1^{-\frac{n-2}{2}-s} L^1_{t,r} H^{s-1/2}_\omega
    + \dot{\ell}^{3/2-s}_2 L^2_t L^2_x}
\end{multline}
  for any $p\in (2,\infty)$, $s\in (1/2-1/p,1/2)$, and $0\le \sigma< \min(s-1/2+1/p,1/2-1/p)$.
\end{coro}

It is the technique to prove Theorem \ref{thm-wStri} that is the
biggest departure from \cite{LMSTW}, and it is these techniques that
permit us to drop the assumption that the data are compactly
supported.  The key new ideas that are being implemented are from
\cite{MetTa07}.  To begin, we have the following localized energy
estimate for small, asymptotically flat metric perturbations.

\begin{thm}[\cite{MetTa07}, Corollary 1]\label{thm-MetTa}
  Let $n\ge 3$ and $-1<\delta<0$.  Under the same
  conditions as in Corollary \ref{cor-wStri}, we have
  \begin{equation}
    \label{MT_LE}
   \|\partial w\|_{L^\infty_t\dot{H}^\delta \cap X^\delta} \lesssim \|\partial
   w(0,\cd)\|_{\dot{H}^\delta} + \|P w\|_{L^1_t\dot{H}^\delta + (X^{-\delta})'}.
  \end{equation}
Here
\[\|f\|_{X^\delta}^2 = \sum_{k=-\infty}^\infty 2^{2k\delta} \|S_k f\|^2_{X_k}\]
with
\[\|f\|_{X_k}= 2^{k/2} \|f\|_{L^2_{t,x}(A_{\le -k})} + \sup_{j\ge -k}
  \||x|^{-1/2} f\|_{L^2_{t,x}(A_j)},\]
$A_j= \R_+\times \{|x|\approx 2^j\}$, $A_{\le -k} = \cup_{j\le -k}
A_j$, and homogeneous Littlewood-Paley projections $S_kf$.
\end{thm}

For the $X^s$ norm, we observe that we have the following.
\begin{lem}\label{thm-MetTa2}
Let $n\ge 2$ and $0<\de<(n-1)/2$. Then
\begin{equation}
\label{eq-MetTa2}
\|r^{-1/2-\de} u\|_{L^2_{t,x} }\les \|u\|_{X^{\de}
},\quad
\|r^{-1/2-\de} u\|_{L^2_{t,x} }\les \|\nabla u\|_{X^{\de-1}}.
\end{equation}
Moreover,
\begin{equation}
\label{eq-MetTa3}
\|u\|_{(X^{\de})'}\les
\|r^{1/2+\de} u\|_{L^2_{t,x} }.
\end{equation}
\end{lem}
\begin{proof}
The first estimate in \eqref{eq-MetTa2} has appeared
in
\cite[Lemma 1]{Ta08} (when $\de=1/2$) and
 \cite{MetTa07} as equation (13).  
The estimate \eqref{eq-MetTa3} follows from the first by duality and has appear
in  \cite{MetTa07} as equation (15).
The second
estimate in \eqref{eq-MetTa2} follows immediately from the first as it
is elementary to show
\beeq\label{eq-add1}\|S_k u\|_{X_k} \lesssim 2^{-k} \|\nabla S_k u\|_{X_k}.\eneq

For the convenience of the reader, we present a proof of
\eqref{eq-add1}.  We write the symbol, $\phi_k(\xi) = \phi(\xi/2^k)$,
of $S_k$ as
\[\phi_k(\xi) = \sum_{i=1}^n \Bigl(2^k \xi_i |\xi|^{-2}
\psi(\xi/2^k)\Bigr) 2^{-k}\xi_i \phi_k(\xi)\]
where $\psi\in C^\infty_0$ is identically $1$ on the support of $\phi$
and vanishes in a neighborhood of $0$.  If we let $\Psi^i_k$ be the
operator with symbol $2^k \xi_i |\xi|^{-2}
\psi(\xi/2^k)$, it suffices to show that $\|\Psi_k^i\|_{X_k\to X_k} =
\O(1)$.

We expand
\[\mathbf{1}_{A_j} \Psi^i_k u = \mathbf{1}_{A_j} \Psi^i_k
\mathbf{1}_{A_{\le -k}} u + \sum_{l>-k} \mathbf{1}_{A_j} \Psi^i_k
\mathbf{1}_{A_l} u.\]
The operators $\mathbf{1}_{A_j} \Psi^i_k
\mathbf{1}_{A_l}$ have kernel
\[K_{jl}(x,y) = 2^{kn} \mathbf{1}_{A_j}(x)\mathbf{1}_{A_l}(y)
a(2^k(x-y))\]
for some Schwartz function $a$.  In the case that $l=-k$, the
indicator function $\mathbf{1}_{A_l}$ is replaced by
$\mathbf{1}_{A_{\le -k}}$.  The analogous substitution is made when $j=-k$.

Now suppose that $j>-k$.  Since for any given $N\gg 1$, 
\[2^{l/2} \|K_{jl}(x,y)\|_{L^\infty_x L^1_y \cap L^\infty_y L^1_x}
\lesssim 
\begin{cases}
2^{l/2}      \ ,&   l-j\le 2\\
2^{-l N}2^{-k(N+1/2)}     \ , &   l-j\ge 3,
\end{cases}
\]
it follows from Young's inequality that
\[\|\mathbf{1}_{A_j}(x) \Psi^i_k\|_{X_k\to L^2} \lesssim \sum_{l\ge
    -k} 2^{l/2} \|K_{jl}\|_{L^\infty_x L^1_y\cap L^\infty_y L^1_x}
  \lesssim 2^{j/2},\]
as desired.

When $j=-k$, we similarly note that
\[2^{l/2} \|K_{(-k)l}(x,y)\|_{L^2_xL^2_y}
\lesssim 
\begin{cases}
2^{-k/2}\, & l=-k,\\
2^{(k+l)n/2} 2^{-lN} 2^{-k(N+1/2)}    \ , &   l>-k.
\end{cases}
\]
Thus, by the Schwarz inequality,
\[\|\mathbf{1}_{A_{\le -k}} \Psi^i_k\|_{X_k\to L^2} \lesssim \sum_{l \ge
  -k} 2^{l/2} \|K_{(-k)l}\|_{L^2_xL^2_y} \lesssim 2^{-k/2},
\]
which completes the proof.

 \end{proof}

We also have
\begin{lem}[\cite{MetTa07}, Lemma 2]\label{thm-MetTa2-2}\noindent
  \begin{enumerate}
  \item[(a)] Suppose that $\|b\|_{\dot{\ell}_1^1 L^\infty_{t,x}} + \|\partial
    b\|_{\dot{\ell}^2_1 L^\infty_{t,x}}\le \delta_1$ and that $|\delta|\le 1$,
    $|\delta|<\frac{n-1}{2}$.  Then
    \begin{equation}
      \label{MetTa2-b}
      \|b\partial u\|_{(X^{-\delta})'} \lesssim \delta_1 \|\partial u\|_{X^\delta}.
    \end{equation}
\item[(b)] Suppose $n\ge 3$, $\|c\|_{\dot{\ell}^2_\infty L^\infty_{t,x}} \le \delta_1$, and
  $-1<\delta<0$.  Then
  \begin{equation}
    \label{MetTa2-c}
   \|cu\|_{(X^{-\delta})'} \lesssim \delta_1 \|\partial u\|_{X^\delta}.
  \end{equation}
  \end{enumerate}
\end{lem}

Finally, we record the following regarding the interpolation of $X^s$
spaces.
\begin{lem}\label{lem-Xinterp}
  For $X^s$ as in Theorem~\ref{thm-MetTa}, we have $[X^{s_1}, X^{s_2}]_\theta  = X^{\theta    s_1 + (1-\theta)s_2}$ and $[(X^{s_1})',(X^{s_2})']_\theta =
  (X^{\theta s_1 + (1-\theta)s_2})'$, for any $\theta\in (0,1)$.
\end{lem}

\begin{proof}
  Fixing $\beta\in C^\infty_0(\R)$ with $\supp\, \beta \subset [1/2,2]$
  and $\sum_{j\in \Z} (\beta(|x|/2^j))^2 = 1$ for any $0\neq x\in \R^n$, we notice that
\[\|u\|^2_{X^s} =   \|\beta(\la 2^k\cd\ra/2^j)
S_k u\|^2_{\dot{\ell}^{s+\frac{1}{2}}_2 \ell^{-\frac{1}{2}}_\infty L^2_{t,x}}\]
where the $\dot{\ell}^{s+\frac{1}{2}}_2$ summation is over $k\in \Z$ and
$\ell^{-\frac{1}{2}}_\infty$ is in $j\ge 0$.  It follows that $X^s$ is a
retract of $\dot{\ell}^{s+\frac{1}{2}}_2 \ell^{-\frac{1}{2}}_\infty
L^2_{t,x}$. 
See, e.g., \cite[Definition 6.4.1]{BL} for the definition of retract. 
  Indeed, the co-retraction and retraction are given by
\[Qf = \Bigl(\beta(\la 2^k x\ra/2^j) S_k f\Bigr)_{j\in \N, k\in
  \Z},\quad Q:X^s\to \dot{\ell}^{s+\frac{1}{2}}_2 \ell^{-\frac{1}{2}}_\infty
L^2_{t,x}\ ,\]
\[R(a_{lm})_{l\in \N,m\in \Z} = \sum_{m\in \Z} S_m \sum_{l\ge 0}
\beta\Bigl(\frac{\la 2^m x\ra}{2^l}\Bigr) a_{lm},\quad
R:\dot{\ell}^{s+\frac{1}{2}}_2 \ell^{-\frac{1}{2}}_\infty L^2_{t,x}\to X^s.\]
Moreover, $RQ f = f$ for $f\in X^s$.  
By \cite[Theorem 5.6.3]{BL}, we know that, with $s=\theta   s_1 + (1-\theta)s_2$,
$$[\dot{\ell}^{s_1+\frac{1}{2}}_2 \ell^{-\frac{1}{2}}_\infty
L^2_{t,x}, \dot{\ell}^{s_2+\frac{1}{2}}_2 \ell^{-\frac{1}{2}}_\infty
L^2_{t,x}]_\theta=
\dot{\ell}^{s+\frac{1}{2}}_2 \ell^{-\frac{1}{2}}_\infty
L^2_{t,x}\ .$$
Then we have
$[X^{s_1}, X^{s_2}]_\theta  = X^{s}$ by 
\cite[Theorem 6.4.2]{BL}.
For the dual estimate,
$[(X^{s_1})', (X^{s_2})']_\theta  = (X^{s})'$, it follows
from \cite[Theorem 4.5.1]{BL}.
\end{proof}

With these results in place, we now proceed to the
proof of the main linear estimate.

\begin{proof}[Proof of Theorem \ref{thm-wStri}]

We first note that, using \eqref{H1}, 
\[ [P, \partial^\mu \Omega^\nu]w = \sum_{\substack{|\tilde{\mu}|+|\tilde{\nu}|\le
    |\mu|+|\nu|\\|\tilde{\nu}|\le |\nu|}} \tilde{b}^\alpha_{\tilde{\mu}\tilde{\nu}} \partial^{\tilde{\mu}} \Omega^{\tilde{\nu}} \partial_\alpha w +
\tilde{c}_{\tilde{\mu}\tilde{\nu}} \partial^{\tilde{\mu}} \Omega^{\tilde{\nu}}w.\]
And given $\delta_1\ll 1$, by \eqref{H1.1} and \eqref{H1.2}, we may fix $R>2R_1$
sufficiently large so that
\[\|\mathbf{1}_{>R} \tilde{b}\|_{\ell^1_1 L^\infty_{t,x}} + \|\mathbf{1}_{>R} (\partial
\tilde{b},\tilde{c})\|_{\ell^2_1 L^\infty_{t,x}} \le \delta_1.\]

We suppose that $Pw = F$.  Thus,
\[P\psi_R \partial^\mu \Omega^\nu w =
[P,\psi_R] \partial^\mu \Omega^\nu w + \psi_R [P,\partial^\mu\Omega^\nu]w +
\psi_R \partial^\mu\Omega^\nu F.\]
Then
with
$n\ge 3$ and $s_1\in (0,1)$, we apply 
Theorem \ref{thm-MetTa} with $\delta=s_1-1$ to get
\begin{multline*}
\|\partial (\psi_R Z^{\le 2} w)\|_{L^\infty_t \dot{H}^{s_1-1}_x \cap X^{s_1-1}}
\les 
\|\psi_R Z^{\le 2} w(0)\|_{\dot H^{s_1}}
+\|\psi_R \pt Z^{\le 2} w(0)\|_{\dot H^{s_1-1}}
\\+\|\psi_R Z^{\le 2} F\|_{L^1_t\dot H^{s_1-1}}
+\|[P,\psi_R] Z^{\le 2}
w\|_{(X^{-(s_1-1)})'} 
\\+ \|\psi_R \tilde{b} Z^{\le 2} \partial w\|_{(X^{-(s_1-1)})'} +
\|\psi_R \tilde{c}
Z^{\le 2} w\|_{(X^{-(s_1-1)})'}.\end{multline*}
Subsequently using Lemma \ref{thm-MetTa2}, we then obtain
\begin{multline*}
\|r^{-1/2-s_1} \psi_R Z^{\le 2}w\|_{L^2_{t,x} }+ \|\partial (\psi_R
Z^{\le 2} w)\|_{X^{s_1-1}}+
\|\psi_R Z^{\le 2} w\|_{L^\infty_t \dot H^{s_1}}
\les 
\|\psi_R Z^{\le 2} w(0)\|_{\dot H^{s_1}}
\\+\|\psi_R \pt Z^{\le 2} w(0)\|_{\dot H^{s_1-1}}
+\|\psi_R Z^{\le 2} F\|_{L^1_t\dot H^{s_1-1}}
+ 
\|r^{3/2-s_1}[P,\psi_R]
Z^{\le 2} w\|_{L^2_{t,x}}
\\+ \|\psi_R \tilde{b} Z^{\le 2} \partial w\|_{(X^{-(s_1-1)})'} +
\|\psi_R \tilde{c}
Z^{\le 2} w\|_{(X^{-(s_1-1)})'}.
\end{multline*}
Moreover,  if $s_2\in(1/2,1)$, we have by the trace estimate \eqref{eq-trace},
\begin{multline*}\|r^{n/2-s_2} \psi_R Z^{\le 2} w\|_{L^\infty_{t,r}
    H^{s_2-1/2}_\omega} + \|\partial (\psi_R Z^{\le 2} w)\|_{X^{s_2-1}}
\les
\|\psi_R Z^{\le 2} w(0)\|_{\dot H^{s_2}}
\\+\|\psi_R Z^{\le 2} \pt w(0)\|_{\dot H^{s_2-1}}
+\|\psi_R Z^{\le 2} F\|_{L^1_t\dot H^{s_2-1}}+
\|r^{3/2-s_2}[P,\psi_R]
Z^{\le 2} w\|_{L^2_{t,x}}
\\+ \|\psi_R \tilde{b} Z^{\le 2} \partial w\|_{(X^{-(s_2-1)})'} +
\|\psi_R \tilde{c}
Z^{\le 2} w\|_{(X^{-(s_2-1)})'}.
\end{multline*}
Then, by interpolation (using Lemma~\ref{lem-Xinterp}), for $p\in (2,\infty)$ and
 $s\in(1/2-1/p,1)$
\begin{multline*}
\|r^{n/2-(n+1)/p-s} \psi_R Z^{\le 2}w\|_{L^p_{t,r}
    H^{\sigma}_\omega}
+ \|\partial \psi_R Z^{\le 2} w\|_{X^{s-1}}
\\\les
\|\psi_R Z^{\le 2} w(0)\|_{\dot H^{s}}
+\|\psi_R Z^{\le 2} \pt w(0)\|_{\dot H^{s-1}}
+\|\psi_R Z^{\le 2}F\|_{L^1_t\dot H^{s-1}}\\
+\|r^{3/2-s}[P,\psi_R] Z^{\le 2} w\|_{L^2_{t,x}}
+ \| \psi_R \tilde{b} Z^{\le 2}  \partial  w\|_{(X^{-(s-1)})'} 
\\+\|\psi_R \tilde{c}
Z^{\le 2} w\|_{(X^{-(s-1)})'}
\end{multline*}
for any $\sigma<\min(s-\frac 1 2+\frac 1p, \frac12-\frac1p)$. 
Here, to obtain the estimate when $s\in (1/2-1/p, 1-2/p]$, we
interpolate with $\theta=1-2/p$, $s_1=\de < \frac{2}{p}\Bigl[s -
\Bigl(\frac{1}{2}-\frac{1}{p}\Bigr)\Bigr]$, 
$s_2=(s-2\de/p)/\theta\in (1/2, 1)$.  Then 
$s=(1-\theta)s_1+\theta s_2$ and $\sigma=\theta(s_2-1/2)=s-1/2+1/p-2\de/p$.
For $s\in [1-2/p, 1)$, with $0<\de <
\Bigl(1-\frac{2}{p}\Bigr)^{-1}(1-s)$, we set $\theta=1-2/p$,
$s_2=1-\de$, $s_1=p(s-\theta+\de \theta)/2\in (0,1)$.  
Then $s=(1-\theta)s_1+\theta s_2$ and $\sigma=\theta(s_2-1/2)=1/2-1/p-\de\theta$.

We have
$$\| \psi_R \tilde{b} Z^{\le 2}  \partial  w\|_{(X^{-(s-1)})'} 
\les
 \|r^{3/2-s} \tilde{b} \psi_R' Z^{\le 2} w\|_{L^2_{t,x}} + 
\|\tilde{b} \partial \psi_R Z^{\le 2} w\|_{(X^{-(s-1)})'} \ ,
$$
and so by Lemma~\ref{thm-MetTa2},
\begin{multline*}
\|r^{n/2-(n+1)/p-s} \psi_R Z^{\le 2}w\|_{L^p_{t,r}
    H^{\sigma}_\omega}
+ \|\partial \psi_R Z^{\le 2} w\|_{X^{s-1}}
\\\les
\|\psi_R Z^{\le 2} w(0)\|_{\dot H^{s}}
+\|\psi_R Z^{\le 2} \pt w(0)\|_{\dot H^{s-1}}
+\|\psi_R Z^{\le 2}F\|_{L^1_t\dot H^{s-1}}
\\
+\|r^{3/2-s}[P,\psi_R] Z^{\le 2} w\|_{L^2_{t,x}}
+ \|r^{3/2-s} \tilde{b} \psi_R' Z^{\le 2} w\|_{L^2_{t,x}} 
\\+ 
\|\tilde{b} \partial \psi_R Z^{\le 2} w\|_{(X^{-(s-1)})'} 
+\|\psi_R \tilde{c}
Z^{\le 2} w\|_{(X^{-(s-1)})'}\ .
\end{multline*}
We may now apply \eqref{MetTa2-b} and \eqref{MetTa2-c} to bootstrap
the last two terms, provided that $\delta_1$ is sufficiently small.
The remainder of the proof is independent of the choice of $R$, and as
such, our implicit constants moving forward may now depend on $R$.

Since $[P,\psi_R]$ and $\psi_R'$ are supported on the fixed annulus $\{|x|\approx
R\}$ and since the coefficients of $Z$ are $\O(1)$ on this annulus, it
follows that
\begin{multline*}\|r^{3/2-s} [P,\psi_R]Z^{\le 2} w\|_{L^2_{t,x}} + \|r^{3/2-s}
\tilde{b}\psi_R' Z^{\le 2} w\|_{L^2_{t,x}}\\\lesssim
\|(1-\chi)\partial \partial^{\le 2} w\|_{\ell^{-1/2}_\infty L^2_{t,x}} +
\|\partial^{\le 2} w\|_{\ell^{-3/2}_\infty L^2_{t,x}}.
\end{multline*}
To these last terms, we may apply \eqref{H2}.
Moreover, as
$\psi_R-\psi^p_R$ is supported on $\{|x|\approx R\}$ and as
$\|f\|_{\dot{H}^{-\delta}}\lesssim \|f\|_{L^2}$ for $f$ supported on a
  fixed ball and  $0\le \delta < \frac{n}{2}$, we have
\[ \|(\psi_R-\psi^p_R) Z^{\le 2} F\|_{L^1_t\dot{H}^{s-1}} \lesssim
\|\partial^{\le 2} F\|_{L^1_tL^2_x}.\]
So,  with $s\in(1/2-1/p,1/2)$, which ensures that $1-s>1/2$, we get by the dual to the trace estimates \eqref{eq-trace},
\begin{multline*}\|r^{n/2-(n+1)/p-s}\psi_R Z^{\le 2} w\|_{L^p_{t,r} H^{\sigma}_\omega}\les \|\psi_R Z^{\le 2} w(0)\|_{\dot H^{s}}
+\|\psi_R Z^{\le 2}\pt w(0)\|_{\dot H^{s-1}}\\ 
+\|w(0)\|_{H^3} + \|\partial_t w(0)\|_{H^2}+\|r^{-(n-2)/2-s}
\psi^p_R Z^{\le 2}F\|_{L^1_{t,r} H^{s-1/2}_\omega}
+\|\partial^{\le 2} F\|_{L^1_t L^2_x}.
\end{multline*}

Finally, we recast the first two terms on the right in terms of our
initial data.  For the first term, the only issue is when both vector
fields are $\partial_t$.  Since $R$ is chosen sufficiently large so
that $P$ is a small perturbation of $\Box$ on the support of $\psi_R$,
we may use the equation to establish
\[\|\psi_R \partial_t^2 w(0)\|_{\dot{H}^s} \lesssim \|\psi_R Y^{\le 2}
w(0)\|_{\dot{H}^s} + \|\psi_R Y^{\le 2} \partial_t
w(0)\|_{\dot{H}^{s-1}} + \|\psi_R Y^{\le 1} F(0)\|_{\dot{H}^{s-1}}.\]
Similarly using the equation to convert occurrences of $\partial_t^3$
and $\partial_t^2$ yields
\begin{multline*}
\|\psi_R Z^{\le 2} w(0)\|_{\dot{H}^s} + \|\psi_R Z^{\le
  2} \partial_t w(0)\|_{\dot{H}^{s-1}}
\lesssim \|\psi_R Y^{\le 2}
w(0)\|_{\dot{H}^s} \\+ \|\psi_R Y^{\le 2} \partial_t
w(0)\|_{\dot{H}^{s-1}} + \|\psi_R Z^{\le 1} F(0)\|_{\dot{H}^{s-1}},
\end{multline*}
and completes the proof.
\end{proof}

\section{Small data global existence}

Here we state a more precise version of Theorem \ref{metaTheorem0} and
provide a proof.  For a given $q$, we shall apply
Theorem~\ref{thm-wStri} with $s=\frac{n}{2}-\frac{2}{q-1}$.  We note
that $s\in (1/2-1/q,1/2)$ precisely
$p_c<q<p_{\mathrm{conf}}:=1+\frac{4}{n-1}$, where the latter is the conformally
invariant exponent above which small data global existence may be
established in the flat case using Strichartz estimates.  See, e.g.,
\cite{Sogge}.  We
set $-\alpha = \frac{n}{2}-\frac{n+1}{q}-s = \frac{2}{q-1} -
\frac{n+1}{q}$ to be the power of the weight in the left side of
\eqref{weightedStrichartz} and note that $-\frac{n-2}{2}-s = -\alpha q$.
We define the norms
\begin{equation}
   \|u\|_{X_k} =  
  \| r^{-\alpha} \psi_R Z^{\le k} u \|_{L^q L^q L^\theta} +
  \| \partial^{\le k} u\|_{\ell^{-3/2}_\infty L^2 L^2 L^2}
+\|\pa^{\le k}\pa u\|_{L^\infty L^2 L^2},\label{X-norm}
\end{equation}
\begin{equation}\label{N-norm}
 \|g\|_{N_k} = \| r^{- \alpha q} \psi_R^q Z^{\le k} g
   \|_{L^1 L^1 L^{2}} +
 \|Z^{\le k} g\|_{L^1 L^2 L^2}
\end{equation}
where $R>R_1$ is sufficiently large as dictated by
Theorem~\ref{thm-wStri}.  Here we may choose any $\theta$ satisfying
\[2< \theta < \min\Bigl(q,\frac{2(n-1)}{n-1-2\min(s-1/2+1/q,1/2-1/q)}\Bigr),\]
where we additionally note that $\min(s-1/2+1/q,1/2-1/q)>0$ provided
$q>\max(2,p_c)$.  In the case of $n=3$, it would suffice to work with
$\theta = 2$.

\begin{thm}\label{main_thm}
Let $n=3,4$, and assume that \eqref{H1}, \eqref{H1.1}, \eqref{H1.2} and \eqref{H2}
hold.  Consider \eqref{main.eq} with $p>p_c$.  Set
$s=\frac{n}{2}-\frac{2}{q-1}\in (1/2-1/q,1/2)$ where $q=p$ if $p\in
(p_c,p_{\mathrm{conf}})$ and $q\in (p_c, p_{\mathrm{conf}})$ is any fixed choice when
$p\ge p_{\mathrm{conf}}$.  Then there exists
$\varepsilon_0>0$ sufficiently small and a $R>R_0$ sufficiently large, 
so that if $0<\varepsilon<\varepsilon_0$ and
\begin{equation}
  \label{data}
  \|Y^{\le 2} \nabla_x^{\le 1} u_0\|_{L^2} + \|\nabla_x^{\le 2}
  u_1\|_{L^2} + \|Y^{\le 2} u_1\|_{\dot{H}^{s-1}}\le \varepsilon,
\end{equation}
then there exists a global solution $u$ with $\|u\|_{X_2}\lesssim \varepsilon$.
\end{thm}

The proof follows that of \cite{LMSTW}
quite closely, and we only highlight the main points.  First of all,
it suffices to assume $p\in (p_c,p_{\mathrm{conf}})$.  If not, one need only
fix any $q\in (p_c,p_{\mathrm{conf}})$ and apply the proof below while noting that Sobolev embeddings provide
$\|u\|_{L^\infty_{t,x}} \lesssim \|u\|_{X_2}$, which suffices to handle
the $p-q$ extra copies of the solution in the nonlinearity.

An iteration is set up in $X_0$.  The estimate
\eqref{weightedStrichartz} can be combined with \eqref{data} and
Sobolev embeddings to show that
\[\|u\|_{X_k} \lesssim \varepsilon + \|F_p(u)\|_{N_k}.\]
The key points to bound the iteration (in $X_2$) and to show that it
converges (in $X_0$)
are to show
\begin{equation}
  \label{boundedness}
  \|F_p(u)\|_{N_2} \lesssim \|u\|^p_{X_2}
\end{equation}
and
\begin{equation}
  \label{Cauchy}
  \|F_p(u)-F_p(v)\|_{N_0} \lesssim \|(u,v)\|_{X_2}^{p-1} \|u-v\|_{X_0}.
\end{equation}

To this end, for functions $f, g$,  we shall show
\begin{equation}
  \label{trilinear1}
  \|g^{p-2} f^2\|_{N_0} \lesssim \|g\|_{X_2}^{p-2} \|f\|^2_{X_1},
\end{equation}
\begin{equation}
  \label{trilinear2}
  \|g^{p-1} f\|_{N_0} \lesssim \|g\|_{X_2}^{p-1} \|f\|_{X_0}.
\end{equation}
To see how these yield \eqref{boundedness} and \eqref{Cauchy}, we
observe that \eqref{Fp} guarantees
\[|Z^{\le 2} F_p(u)| \lesssim |u|^{p-1} |Z^{\le 2} u| + |u|^{p-2}
|Z^{\le 1} u|^2.\]
Thus an application of \eqref{trilinear2} and \eqref{trilinear1} to the terms
in the right side respectively yields \eqref{boundedness}.  Similarly,
\eqref{Fp} shows that
\[|F_p(u) -F_p(v)| \lesssim |(u,v)|^{p-1} |u-v|,\]
which shows that \eqref{Cauchy} is a direct consequence of \eqref{trilinear2}.

Assuming \eqref{trilinear1} and \eqref{trilinear2},
the remainder of the proof of Theorem~\ref{main_thm}
is exactly as in \cite{LMSTW}.
In fact, by setting $u^{(0)}\equiv 0$ and recursively define $u^{(m+1)}$
to be the solution to the linear equation
$$
\begin{cases}
P u^{(m+1)} = F_p(u^{(m)}),\quad (t,x)\in
M\\
u^{(m+1)}(0,x)=u_0(x),\quad \partial_t u^{(m+1)}(0,x)=u_1(x),
\end{cases}
$$
we have that
$$\|u^{(m+1)}\|_{X_2}\lesssim \ep+\|F_p(u^{(m)})\|_{N_2}
\lesssim\ep+ \|u^{(m)}\|_{X_2}^p\ ,
$$ which yields the uniform boundedness for the iteration sequence,
$\|u^{(m+1)}\|_{X_2}\lesssim\ep$ provided that $\varepsilon$ is
sufficiently small. Then
 we get
\begin{align*}\|u^{(m+2)}-u^{(m+1)}\|_{X_0}&\lesssim
\|F_p(u^{(m+1)})-F_p(u^{(m)})\|_{N_0}\\&\lesssim
\|(u^{(m+1)},u^{(m)})\|_{X_2}^{p-1}\|u^{(m+1)}-u^{(m)}\|_{X_0}
\end{align*}
and so the sequence converges in $X_0$, provided that
$\ep\ll 1$.  The limit $u\in X_2$ with $\|u\|_{X_2}\lesssim \ep$ is
the solution we are looking for. 

We shall end by outlining the proof of \eqref{trilinear1} and
\eqref{trilinear2}.
We first observe that 
\[\|r^{-\alpha p} \psi_R^p\, g^{p-1}  f\|_{L^1L^1L^2} 
\lesssim \|r^{-\alpha} \psi_R g\|^{p-1}_{L^pL^pL^\infty} \|r^{-\alpha}\psi_R f\|_{L^pL^pL^2}
\]
and
\[\|r^{-\alpha p} \psi_R^p \, g^{p-2}  f^2\|_{L^1L^1L^2} 
\lesssim \|r^{-\alpha} \psi_R g\|^{p-2}_{L^pL^pL^\infty} \|r^{-\alpha}\psi_R f\|^2_{L^pL^pL^4}.
\]
The Sobolev embeddings $H^2_\omega\subset L^\infty_\omega$ and
$H^1_\omega\subset L^4_\omega$ on $\mathbb{S}^{n-1}$, $n\le 4$ then
show that these are dominated by $\|g\|^{p-1}_{X_2} \|f\|_{X_0}$ and
$\|g\|^{p-2}_{X_2} \|f\|_{X_1}^2$ respectively.

It remains to consider the $L^1L^2L^2$ portions of the $N_0$ norm.
When $r\ge 2R+1$, we use a related argument that relies on the
weighted Sobolev inequalities.  To begin, we note that
\[\|g^{p-1}  f\|_{L^1L^2_{\ge 2R+1}L^2} \lesssim
\|r^{\frac{\alpha}{p-1}}  g\|_{L^p L^{\frac{2p(p-1)}{p-2}}_{\ge 2R+1}
      L^\infty}^{p-1}
\|r^{-\alpha} f\|_{L^p L^p_{\ge 2R+1} L^2}.\]
Noting also that
\[\frac{\alpha}{p-1} -\frac{n-1}{p} + \frac{(n-1)(p-2)}{2p(p-1)} \le
-\alpha\ \text{ for }\ p\le p_{\mathrm{conf}}.\]
and applying \eqref{eq-Sobo} establishes that
is controlled by $\|g\|^{p-1}_{X_2} \|f\|_{X_0}$.
We note that this is the crucial place where we use the
$L^\theta_\omega$ norm with $\theta>2$ in order to stay within the
allowed regularity when $n=4$.\footnote{When $n=3$, it would suffice
  to work with $\theta=2$.  An alternate approach to the $n=4$ case
  would be to note that the $\delta=-1$ analogs of
  Theorem~\ref{thm-MetTa}, Lemma~\ref{thm-MetTa2}, and
  Lemma~\ref{thm-MetTa2-2} are proved for $n=4$ in \cite{MetTa07}.
  This suffices to prove an estimate akin to \cite[Theorem 2.3]{DMSZ}
  from which the corresponding $n=4$ existence theorem was shown to
  follow.}
Similarly, noting that
\[\frac{2}{p-2}\Bigl(\alpha - \frac{n-1}{p}+\frac{n-1}{4}\Bigr)\le
-\alpha + \frac{n-1}{p}\ \text{ for } \ p\le p_{\mathrm{conf}},\]
we may apply \eqref{eq-Sobo}, \eqref{eq-Sobo2}, and H\"older's
inequality to bound
\begin{multline*}
\|g^{p-2} \, f^2\|_{L^1 L^2_{\ge 2R+1} L^2} \\\lesssim
\|r^{\frac{2}{p-2}\left(\alpha-\frac{n-1}{p}+\frac{n-1}{4}\right)}g\|_{L^p L^\infty_{\ge 2R+1} L^\infty}^{p-2}
\|r^{-\alpha +  \frac{n-1}{p}-\frac{n-1}{4}} f\|^2_{L^p L^4_{\ge 2R+1} L^4} 
\end{multline*}
by $\|g\|^{p-2}_{X_2} \|f\|^2_{X_1}$.

Finally we need to control the same over $B_{2R+1}$ where the vector
fields $Z$ all have bounded coefficients. 
 In $n=3,4$, we claim that we have 
\beeq\label{eq-add2}\|g\|_{L^\infty_x} \lesssim \|g\|^{\frac{1}{p-1}}_{H^2}
\|g\|^{1-\frac{1}{p-1}}_{\dot{H}^1\cap\dot{H}^3}\ ,\eneq
for any $p>p_c$,
and so
\begin{equation}\label{eng-terms}\|g\|_{L^{2(p-1)}L^\infty_{\le 2R+1} L^\infty} \lesssim
\|g\|^{\frac{1}{p-1}}_{L^2_t H^2_{\le 2R+2}}
\|g\|^{1-\frac{1}{p-1}}_{L^\infty_t(\dot{H}^3\cap \dot{H}^1)}\lesssim
\|g\|_{X_2}.
\end{equation}
As
\[\|g^{p-1} \, f\|_{L^1 L^2_{\le 2R+1} L^2} \lesssim
\|g\|^{p-1}_{L^{2(p-1)}L^\infty_{\le 2R+1} L^\infty} \|
f\|_{L^2 L^2_{\le 2R+1}L^2},\]
it follows from \eqref{eng-terms} that the right side is controlled by
$\|g\|_{X_2}^{p-1} \|f\|_{X_0}$, which completes the proof of \eqref{trilinear2}.  

Similarly, as
\[\|g^{p-2} \, f^2\|_{L^1 L^2_{\le 2R+1} L^2} \lesssim
\|g\|^{p-2}_{L^\infty L^\infty L^\infty} \|f\|^2_{L^2
  L^4_{\le 2R+1} L^4},\]
the Sobolev embeddings $\dot{H}^3\cap \dot{H}^1\subset L^\infty$,
$H^1\subset L^4$ show that this is also controlled by $\|g\|_{X_2}^{p-2}
\|f\|^2_{X_1}$, which finishes the proof of \eqref{trilinear1}. 

For the proof of \eqref{eq-add2}, we recall that if $0<a<b<c\le
\infty$ and $\frac{1}{b} = \frac{\lambda}{a} + \frac{1-\lambda}{c}$,
then
\[\|g\|_{L^b} \le \|g\|_{L^a}^\lambda \|g\|_{L^c}^{1-\lambda}.\]
We shall apply this with $a=2$, $c=2n/(n-2)$, and $\lambda =
\frac{1}{p-1}$.  This yields $b = 2n(p-1)/(2+(n-2)(p-1))$ and $b>2$ if
$p>2$.  Hence, by standard Sobolev embeddings, we obtain the
following, which implies \eqref{eq-add2},
\begin{align*}\|g\|_{L^\infty}&\lesssim \|\nabla_x^{\le 2}
g\|_{L^{\frac{2n(p-1)}{2+(n-2)(p-1)}}} \lesssim \|\nabla_x^{\le 2}
g\|^{\frac{1}{p-1}}_{L^2} \|\nabla_x^{\le 2}
g\|^{1-\frac{1}{p-1}}_{L^{\frac{2n}{n-2}}}\\ &\lesssim \|\nabla_x^{\le 2}
g\|^{\frac{1}{p-1}}_{L^2} \|\nabla_x^{\le 2}
g\|^{1-\frac{1}{p-1}}_{\dot{H}^1}.
\end{align*}




\end{document}